\documentclass[11pt,leqno]{amsart}
\usepackage{amscd}
\usepackage{color}
\usepackage[symbol]{footmisc}
\usepackage{amssymb}
\usepackage{amsfonts}
\usepackage{latexsym}
\usepackage{verbatim}
\usepackage{bbm}

\newcommand{\R}{\mathbb{R}}
\newcommand{\C}{\mathbb{C}}

\renewcommand{\epsilon}{\varepsilon}
\renewcommand{\theta}{\vartheta}
\renewcommand{\phi}{\varphi}
\renewcommand{\Re}{\mathrm{Re}}
\renewcommand{\Im}{\mathrm{Im}}
\newcommand\norma[1]{\left\lVert#1\right\rVert}

\theoremstyle{plain}
\newtheorem{teor}{Theorem}
\newtheorem{prop}[teor]{Proposition}
\newtheorem{lem}[teor]{Lemma}

\theoremstyle{definition}

\theoremstyle{definition}

\title{Fully non-linear elliptic equations on compact hyperk\"ahler manifolds}

\begin{document}
	
\thanks{This work was supported by GNSAGA of INdAM}

\address{(Giovanni Gentili) Dipartimento di Matematica G. Peano \\ Universit\`a degli Studi di Torino\\
		Via Carlo Alberto 10\\
		10123 Torino\\ Italy.}
\email{giovanni.gentili@unito.it}
	
\address{(Luigi Vezzoni) Dipartimento di Matematica G. Peano \\ Universit\`a degli Studi di Torino\\
Via Carlo Alberto 10\\
10123 Torino\\ Italy.}
\email{luigi.vezzoni@unito.it}

	
	
\author{Giovanni Gentili and Luigi Vezzoni}
	
\date{\today}

\maketitle
\begin{abstract}
We consider a general class of elliptic equations on hypercomplex manifolds which includes the quaternionic Monge-Amp\`ere equation, the quaternionic Hessian equation and the Monge-Amp\`ere equation for quaternionic $(n-1)$-plurisubharmonic functions. We prove that under suitable assumptions the solutions to these equations on hyperk\"ahler manifolds satisfy a $C^{2,\alpha}$ a priori estimate. 
\end{abstract}
	
\section{Introduction}
In the present paper we study a general class of elliptic equations on compact hyperhermitian manifolds. 
The interest in this class of equations moves from a Calabi-Yau--type conjecture on HKT manifolds stated by Alesker and Verbitsky in \cite{Alesker-Verbitsky (2010)} and from the work of Harvey and Lawson \cite{HL1,HL2} about a general class of   Dirichlet problems on special Riemannian manifolds including hyperhermitian manifolds. The conjecture of Alesker and Verbitsky  
states that is always possible to prescribe the $J$-anti-invariant part of the Chern-Ricci form on compact HKT manifolds and has strong applications on the geometry of hyperhermitian manifolds. In analogy to the complex case, the conjecture can be analytically reformulated in terms of a Monge-Amp\`ere--type equation (called {\em quaternionic Monge-Amp\`ere equation}). So far the solvability of such equation is proved only under extra assumptions. In \cite{Alesker (2013)} Alesker proved that the equation is always solvable on compact flat hyperk\"ahler manifolds. The result was drastically improved by Dinew and Sroka in \cite{Dinew-Sroka} who confirmed the conjecture  on every compact hyperk\"ahler manifold (for other results related to the study of the equation see \cite{Alesker-Shelukhin (2017),GV,GV2,Sroka,Srokasharp} and the references therein). Beside the quaternionic Monge-Amp\`ere equation, other parabolic and elliptic equations on hyperhermitian manifolds have been considered in the literature. In \cite{BGV,BGV2,Z} the parabolic counterpart of the quaternionic Monge-Amp\`ere equation was studied, while more general parabolic and elliptic equations are taken into account in   \cite{GZ1,GZ2,HL1,HL2,Srokasharp}. More recently, in \cite{FXZ} a quaternionic analogue of the Monge-Amp\`ere equation for $(n-1)$-plurisubharmonic functions \cite{Form-type,TW17,TW19} is solved on compact hyperkahler manifolds. 

\medskip

A {\em hypercomplex manifold} is a smooth manifold $M$ equipped with three complex structures $(I,J,K)$ satisfying the quaternionic relations
$$
IJ=-JI=K\,.
$$ 
A Riemannian metric $g$ on $M$ is hyperhermitian if it is compatible with each complex structure. A hyperhermitian metric $g$ induces the fundamental form $ \omega(\cdot,\cdot):=g(I\cdot ,\cdot)$. Due to the quaternionic relations $\omega$ is $J$-anti-invariant, i.e. 
$\omega(J\cdot,J\cdot)=-\omega(\cdot,\cdot)$. On the other hand a skew-symmetric $2$-form 
which is compatible with $I$ and is $J$-anti-invariant induces a hyperhermitian metric canonically. A  hypercomplex manifold with a hyperhermitian metric is called a {\em hyperhermitian manifold}.  
 The standard model of hyperhermitian manifold is the quaternionic vector space $\mathbb H^n$ with the Euclidean metric. Moreover, any hyperk\"ahler manifold is in particular hyperhermitian. 
 
\medskip 
Next we describe the class of equations we consider in the present paper. \\
Let $(M^{4n},I,J,K,g)$ be a hyperhermitian manifold of real dimension $ 4n $. Let also $\chi$ be a $ J $-anti-invariant real form in $\Lambda_I^{1,1}M$. For any function $ \phi \colon M \to \R $ which is at least of class $ C^2 $, the form
\[
\chi_\phi:=\chi+\frac{\sqrt{-1}\partial \bar \partial \phi -\sqrt{-1}J\partial \bar \partial \phi}{2}
\]
is a $ J $-anti-invariant real form in $\Lambda_I^{1,1}M$. We denote with $g_\phi$ the corresponding symmetric $2$-tensor. Composing with $ g^{-1} $ we get an endomorphism 
\[
A_\phi:=g^{-1}g_\phi \colon T^{1,0}_IM \to T^{1,0}_IM
\]
which is hyperhermitian with respect to $ g $. We shall study equations of the form
\begin{equation}\label{eq_main}
F(A_\phi)=h\,,
\end{equation}
where $ h \colon M \to \R $ is a smooth datum.  By $F(A_\phi)$ we denote an expression of the type $f(\lambda(A_\phi))$, where $\lambda(A_\varphi)=(\lambda_1,\dots,\lambda_{n})$ denotes the eigenvalues of $A_\varphi$ regarded as an $n\times n$ quaternionic matrix and  $f$ is a symmetric function satisfying the following conditions:
\begin{enumerate}
\item[1.] $ f $ is defined on a symmetric proper convex open cone $ \Gamma $ in $ \R^{n} $ with vertex at the origin and containing the positive orthant $ \Gamma_{n}:=\{ \lambda \in \R^{n} \mid \lambda_i>0\,, \forall i=1,\dots,n \} $.

\vspace{0.1cm}
\item[2.] $ f_i:=\frac{\partial f}{\partial \lambda_i}>0 $ for all $ i=1,\dots,n $ and $ f $ is a concave function.

\vspace{0.1cm}
\item[3.] $ \sup_{\partial \Gamma}f<\inf_Mh $, where $ \sup_{\partial \Gamma}f=\sup_{\lambda_0\in \partial \Gamma} \limsup_{\lambda \to \lambda_0} f(\lambda) $.

\vspace{0.1cm}
\item[4.] For any $ \sigma <\sup_\Gamma f $ and $ \lambda \in \Gamma $ we have $ \lim_{t\to \infty} f(t\lambda)>\sigma $.
\end{enumerate}
Assumption 2. ensures that if $\phi$ is {\em $\Gamma$-admissible}, i.e.
\[
\lambda\left( A_\phi\right)\in \Gamma\,,
\]
the equation is elliptic, while 3. guarantees non-degeneracy of the equation and then uniform ellipticity once the second order estimate is obtained.
	
\medskip 	
The same framework has been investigated at length and has a long lasting tradition starting from the influential paper of Caffarelli, Nirenberg and Spruck \cite{CNS}, where the Dirichlet problem on domains of $ \R^n $ is considered. The setting we described above is the natural generalization to the quaternionic case of the one considered by Sz\'{e}kelyhidi in \cite{Szekelyhidi} for studying fully non-linear elliptic equations on complex manifolds. Note that the quaternionic Monge-Amp\`ere equation, the quaternionic Hessian equation \cite{GZ1,HL1,HL2} and the Monge-Amp\`ere equation for quaternionic $(n-1)$-plurisubharmonic functions \cite{GZ1} belong to this general class of equations.

\medskip
Our main result is the following:
	
\begin{teor}\label{main}
Let $ (M^{4n},I,J,K,g) $ be a compact hyperk\"ahler manifold, $ \chi\in \Lambda^{1,1}_IM $ a real $ J $-anti-invariant form, and $ \underline{\phi} $ a $ \mathcal{C} $-subsolution of \eqref{eq_main}. Then there exist $ \alpha \in (0,1) $ and a constant $ C>0 $, depending only on $ (M,I,J,K,g),$  $\chi,$  $h $ and $ \underline{\phi} $, such that any $\Gamma$-admissible solution $ \phi $ to \eqref{eq_main} with $ \sup_M\phi=0 $ satisfies the estimate
\[
\|\phi\|_{C^{2,\alpha}}\leq C\,.
\]
\end{teor}
	
In the statement by {\em $ \mathcal{C} $-subsolution of \eqref{eq_main}} we mean that $ \underline{\phi}\in C^2(M,\R) $ is such that 
$$
\left( \lambda\bigl(A_{\underline{\phi}}\bigr) +\Gamma_n \right)\cap \partial \Gamma^{h(x)}	\mbox{ is bounded for every }x\in M\,,
$$
where for every  $ \sigma>\sup_{\partial \Gamma} f $, $ \Gamma^\sigma$ denotes the convex superlevel set $ \Gamma^\sigma=\{ \lambda \in \Gamma \mid f(\lambda)>\sigma \} $. 

\medskip 
The proof of Theorem \ref{main} is obtained as follows:

\vspace{0.1cm}
In section \ref{C0} we prove that solutions to \eqref{eq_main} satisfy a $C^0$-a priori bound. This section is quite general since the hyperk\"ahler assumption does not play a role and the estimate we obtain holds true when the manifold is simply hyperhermitian. 

\vspace{0.1cm}
In section \ref{Laplacian} we prove that the Laplacian of solutions to \eqref{main} satisfies the following estimate
$$
\Delta_g \phi \leq C\left( \|\nabla \phi \|_{C^0}+1 \right)
$$
for a positive constant $ C $ depending on the data. We use an approach introduced by Chou and Wang \cite{Chou-Wang} to study the Hessian equation (see also Hou, Ma and Wu \cite{Hou-Ma-Wu}). Here is where the hyperk\"ahler assumption plays a role. A key observation is that concavity of the equation implies that $F$ satisfies 
$$
F^{r\bar s}(A)A_{r\bar s}\leq C\,\sum_{k=1}^{2n}F^{k\bar k}(A)
$$
for every hyperhermitian matrix $A$ such that $\lambda(A)\in \partial \Gamma^\sigma $, where the constant $ C $ depends on $ \sigma \in [\sup_{\partial \Gamma} f, \sup_\Gamma f] $. Here $F^{r\bar s}$ are the first derivatives of 
$F$ with respect to the $(r,\bar s)$-th entry. Note that if $F_{*}$ denotes the differential of $F$ we have 
$$
F_{*|A}(X):=F^{i\bar j}(A)X_{i\bar j}
$$
and $\sum_{k=1}^{2n}F^{k\bar k}(A)$ is the trace of gradient of $F$ once it is regarded as a matrix. Note that our Laplacian estimate is shaper than the one obtained by Sz\'{e}kelyhidi \cite[Proposition 13]{Szekelyhidi}, since Sz\'{e}kelyhidi's estimate involves $\|\nabla \phi \|_{C^0}^2$, while the hyperk\"ahler assumption allows us to prove an estimate which involves $\|\nabla \phi \|_{C^0}$ only. 

\vspace{0.1cm}
In section \ref{gradient} we prove the $C^1$-estimate. The fact that our Laplacian estimate involves $\|\nabla \phi \|_{C^0}$, only, allows us to obtain the $C^1$-estimate quite easily. Indeed we can combine an interpolation inequality, Morrey's inequality and elliptic bounds in order to prove the estimate. In particular we do not need to apply any Liouville-type theorem.  

\vspace{0.1cm} In section \ref{C2alpha} we obtain the $C^{2,\alpha}$-estimate by using a general result of 
Tosatti, Wang, Weinkove and Yang \cite{TWWY} and in section \ref{finale} we combine the results of the previous sections in order to prove Theorem \ref{main}.  

\bigskip
	
\noindent {\bf Acknowledgements.}
The first-named author is grateful to Elia Fusi for many useful discussions.

%
%
%
	
\section{$ C^0 $ estimate}\label{C0}
In this section we show that every solution $\varphi$ of class $C^2$ to \eqref{eq_main} satisfies a $C^0$ a priori bound. The strategy we adopt is to use the Alexandrov-Bakelman-Pucci (ABP for short) maximum principle in the form of \cite[Proposition 10]{Szekelyhidi}. Such an idea can be traced back to the work of B\l ocki \cite{Blocki} for the complex Amp\`ere equation. For a more restrictive class of equations Sroka proves in  \cite{Srokasharp} a sharp 
$ C^0 $-estimate by adapting a technique of Guo and Phong \cite{GP22a,GP22b} and of Guo, Phong and Tong \cite{GPT21} to the quaternionic case. The class of  equations considered by Sroka is more restrictive than the one taken into account in the present paper, on the other hand Sroka's estimate is sharper than ours.

\begin{lem}\label{Lem:Harnack}
Let $ (M^{4n},I,J,K,g) $ be a compact hyperhermitian manifold. If $ \phi $ is a solution to \eqref{eq_main}, then there exist $ p,C>0 $, depending only on the background data, such that
\[
\left\|\phi-\sup_M \phi\right\|_{L^p}\leq C\,.
\]
\end{lem}
\begin{proof}
The proof is a standard application of the weak Harnack inequality. We explain the main ideas for convenience of the reader. Take an open cover of $ M $ made of coordinate balls $ 2B_i:=B_{2r_i}(x_i) $ such that the balls $ B_i $ with half the radius still cover $ M $. Since 
$$
\Gamma \subseteq \left\{ \lambda \in \R^{n} \mid \sum_{i=1}^{n}\lambda_i>0 \right\}\,,
$$
we have $ \mathrm{tr}_\omega \chi_\phi >0 $ and so, using that $ \omega $ is $J$-anti-invariant, we get
\[
\Delta_\omega \phi=\mathrm{tr}_\omega(\sqrt{-1}\partial \bar \partial \phi)=\mathrm{tr}_\omega\left(\frac{\sqrt{-1}\partial \bar \partial \phi-\sqrt{-1}J\partial \bar \partial \phi}{2}\right)=\mathrm{tr}_\omega \chi_\phi- \mathrm{tr}_\omega \chi\geq - C\,.
\]
Therefore we can apply the weak Harnack inequality \cite[Theorem 9.22]{GT} to $\psi:=\phi-\sup_{M}\phi$ on $ 2B_i $ deducing
\begin{equation}\label{weakHarnack}
\|\psi\|_{L^p(B_i)}\leq C\left( \inf_{B_i}(-\psi)+1 \right)\,,
\end{equation}	
where $ p,C>0 $ depend only on the choice of the cover and the background metric. Since $ \psi\leq 0 $ we have $ \inf_{B_j}(-\psi)=0 $ for at least one index $ j $, and thus $ \|\psi\|_{L^p(B_j)}\leq C $. This bound also gives an estimate for $ \inf_{B_i}(-\psi) $ on all coordinate balls intersecting $B_j$. We can then iterate the argument 
by using \eqref{weakHarnack} and obtain an upper bound on each ball of the cover.
\end{proof}

\begin{prop}\label{Prop:C0}
Let $ (M^{4n},I,J,K,g) $ be a compact hyperhermitian manifold. If $ \underline \phi,\phi $ are a $ \mathcal{C} $-subsolution and a solution to \eqref{eq_main} respectively, with $ \sup_M \phi=0 $, then there is a constant $ C>0 $, depending only on the background data and the subsolution $ \underline{\phi} $, such that
\[
\|\phi\|_{C^0}\leq C\,.
\]
\end{prop}
\begin{proof}
Without loss of generality we may assume $ \underline{\phi}=0 $, since we can always modify $ \chi $ in order to obtain $ \underline{\phi}=0 $. Since we assumed $ \sup_M \phi=0 $, the claim is equivalent to a lower bound for $ S=\inf_M \phi $, hence, we may assume $S\leq -1$.
		
Since $ \underline{\phi}=0 $ is a $ \mathcal{C} $-subsolution there are  $ \delta,R>0 $ such that
\begin{equation}\label{eqC0_1}
\left( \lambda\left( g^{-1}\chi\right)-\delta {\bf 1}+\Gamma_n \right)\cap \partial  \Gamma^{h(x)}\subseteq B_R(0)\,, \qquad \text{at every }x\in M\,,
\end{equation}
where $ {\bf 1}=(1,1,\dots,1) $. Pick $ I $-holomorphic coordinates $ (z^1,\dots,z^{2n}) $ centered at the point where $ \phi $ attains its minimum $ S $. We may identify such coordinate neighborhood with the open ball of unit radius $ B_1=B_1(0)\subseteq \C^{2n} $ centered at the origin. Let $ v(x)=\phi(x)+\epsilon|x|^2 $ be defined on $ B_1 $ for some small fixed $ \epsilon>0 $. Observe that $ \inf_{B_1} v=v(0)=\phi(0)=S $ and $ \inf_{\partial B_1}v\geq v(0)+\epsilon $. These conditions allow us to apply the ABP method (see \cite[Proposition 10]{Szekelyhidi}) to obtain
\begin{equation}
\label{eqC0_2}
C_0\epsilon^{4n}\leq \int_P \det(D^2v)\,,
\end{equation}
where $ C_0>0 $ is a constant depending on the dimension of $M$ only,
\[
P=\left\{ x\in B_1\mid |Dv(x)|<\frac{\epsilon}{2},\, v(y)\geq v(x)+Dv(x)\cdot (y-x) \text{ for all }y\in B_1 \right\}\,,
\]
and $ Dv $, $ D^2v $ are the gradient and the (real) Hessian of $ v $. Note that $ P\subseteq \{ x\in B_1 \mid D^2v(x)\geq 0  \} $, then, thanks to a calculation by B\l{}ocki \cite{Blocki},  we have
\begin{equation}\label{eqC0_3}
\det(D^2v) \leq 2^{4n}\det(\mathrm{Hess}_\C v)^2\,, \qquad \text{at every }x\in P\,.
\end{equation}
Applying \cite[Lemma 3.1]{Srokasharp} we also have
\begin{equation}\label{eqC0_6}
\det(\mathrm{Hess}_\C v)\leq 2^{2n} \det\left(    \frac{\mathrm{Hess}_\C v+J^t\mathrm{Hess}_\C v J}{2}\right)\,.
\end{equation}
Furthermore, since convexity implies plurisubharmonicity, we have $ \mathrm{Hess}_\C v(x)\geq 0 $ at any point $ x\in P $ and thus also $ \mathrm{Hess}_\C \phi(x)\geq -\epsilon {\rm Id}$, where $ {\rm Id}$ is the $2n \times 2n $ identity matrix. 
Choosing $ \epsilon $ small enough depending on $ g $ and $ \delta $, we have
\begin{equation}\label{eqC0_4}
\lambda\left( g^{-1}\left(\chi+\frac{\mathrm{Hess}_\C \phi+J^t\mathrm{Hess}_\C \phi J}{2}\right)\right)\in \lambda \left( g^{-1}\chi \right)-\delta {\bf 1}+\Gamma_n\,, \qquad \text{at every }x\in P\,.
\end{equation}
On the other hand, since $ \phi $ solves equation \eqref{eq_main} we also have
\begin{equation}\label{eqC0_5}
\lambda\left( g^{-1}\left(\chi+\frac{\mathrm{Hess}_\C \phi +J^t\mathrm{Hess}_\C \phi J}{2}\right)\right)\in \partial \Gamma^{h(x)}\,, \qquad \text{at every }x\in P\,.
\end{equation}
Together \eqref{eqC0_1}, \eqref{eqC0_4} and \eqref{eqC0_5} imply that $ |\mathrm{Hess}_\C \phi+J^t\mathrm{Hess}_\C \phi J|\leq C $ on $ P $ and thus also $ \mathrm{Hess}_\C v+J^t\mathrm{Hess}_\C v J \leq C $. Consequently, from \eqref{eqC0_2}, \eqref{eqC0_3} and \eqref{eqC0_6} we get
\[
C_0\epsilon^{4n}\leq C\mathrm{Vol}(P)
\]
By definition of $ P $ we have $ v(0)\geq v(x)-Dv(x)\cdot x>v(x)-\epsilon/2 $, i.e. $ v(x)<S+\epsilon/2<0 $ for all $ x\in P $. As a consequence for any $ p>0 $
\[
\|v\|^{p}_{L^p(M)}\geq \|v\|^{p}_{L^p(P)}=\int_P (-v)^p \geq 
\left \lvert S+\frac{\epsilon}{2}\right \rvert^p\mathrm{Vol}(P)\geq C^{-1}C_0\epsilon^{4n}\left  \lvert S+\frac{\epsilon}{2}\right \rvert^p  \,.
\]
Applying Lemma \ref{Lem:Harnack} we find a $ p>0 $ such that $ \|v\|_{L^p} $ is bounded, therefore we conclude.
\end{proof}

\section{Laplacian estimate}\label{Laplacian}
In this section we establish the upper bound of the Laplacian of solutions to \eqref{main}:
	
\begin{prop} \label{Prop:C2}
Let $ (M^{4n},I,J,K,g) $ be a compact hyperk\"ahler manifold. Let $ \underline \phi,\phi $ be a $ \mathcal{C} $-subsolution and a solution to \eqref{eq_main} respectively. Then there is a constant $ C>0 $, depending only on $ (M,I,J,K,g) $, $ \|h\|_{C^2} $, $ \|\chi \|_{C^2} $, $ \|\phi \|_{C^0} $ and $ \underline{\phi} $, such that
\[
\Delta_g \phi \leq C\left( \|\nabla \phi \|_{C^0}+1 \right)\,.
\]
\end{prop}

In order to prove Proposition \ref{Prop:C2} we need the following lemma whose proof is completely analogous to the one of \cite[Proposition 6]{Szekelyhidi}:

\begin{lem}\label{dichotomy}
Let $a,b \in \R$ be such that $ \sup_{\partial \Gamma}f<a<b<\sup_{\Gamma} f $ and let $ \delta, R>0 $. Then there exists a constant $ \kappa>0 $ such that for any $ \sigma \in [a,b] $, every hyperhermitian matrix $ B $ satisfying
\[
\left( \lambda(B)-2\delta {\bf 1}+\Gamma_n \right)\cap \partial \Gamma^\sigma \subseteq B_R(0)\,,
\]
and every hyperhermitian matrix $A $ satisfying $ \lambda(A)\in \partial \Gamma^\sigma $ and $ |\lambda(A)|>R $, we have either
\[
F^{j\bar k}(A)\left( B_{j\bar k}-A_{j\bar k} \right)>\kappa \sum_{k=1}^{2n} F^{k\bar k}(A)
\]
or
\[
F^{j\bar j}(A)>\kappa \sum_{k=1}^{2n} F^{k\bar k}(A)\,, \qquad \text{for all } j=1,\dots,2n\,.
\]
\end{lem}

Now we are ready to prove Proposition \ref{Prop:C2}:
\begin{proof}[Proof of Proposition $\ref{Prop:C2}$]
Consider the quantity
\[
Q=2\sqrt{\lambda_1}+\alpha(|\nabla \phi |^2_g) + \beta(\phi)
\]
where $ \lambda_1 \colon M \to \R $ is the largest eigenvalue of the matrix $ A_\phi $ and
\[
\alpha,\beta \colon \R \to \R \,, \qquad \alpha(t)=-\frac{1}{2}\log \left( 1- \frac{t}{2N} \right)\,, \qquad \beta(t)=-2Dt+\frac{1}{2}t^2\,,
\]
being $N=\|\nabla \phi \|^2_{C^0}+1$ and $ D>\|\phi\|_{C^0} $ a large constant we will determine later. Note that
\begin{align*}
(4N)^{-1}&\leq \alpha'(|\nabla \phi|^2_g) \leq (2N)^{-1}\,, & \alpha''&=2(\alpha')^2\,,\\
-3D &\leq \beta'(\phi) \leq -D\,, & \beta''&\equiv 1\,.
\end{align*}
Let $ x_0\in M $ be a maximum point of $ Q $. In order to prove the statement, it is enough to show that 
\begin{equation}\label{atx0}
\lambda_1(x_0)\leq C\left( \|\nabla \phi \|_{C^0}^2+1 \right)
\end{equation}
for a positive constant $ C $. Indeed, since  
\begin{equation}\label{eq:glob}
2\sqrt{\lambda_1(x)}+\alpha(|\nabla \phi(x) |^2_g) + \beta(\phi(x)) \leq 2\sqrt{\lambda_1(x_0)}+\alpha(|\nabla \phi(x_0) |^2_g) + \beta(\phi(x_0)) 
\end{equation}
at every point $x\in M$ 
and 
$$
0\leq \alpha(|\nabla \phi|^2_g)\leq \frac{1}{2}\log 2\,,
$$ 
then if 	\eqref{atx0} is true we can conclude that 
\[
\lambda_1(x)\leq C'\left( \|\nabla \phi \|_{C^0}^2+1 \right)\,, \quad \mbox{ for every }x\in M\,,
\]
where $ C' $ depends on $ C $, $ D $ and $ \|\phi\|_{C^0} $.

\medskip 
The function $ \lambda_1 $ could be non smooth near $x_0$. In order to overcome this problem we modify its definition as follows. Consider $I$-holomorphic coordinates around $x_0$ such that 
\begin{itemize}
\item the coordinates are normal with respect to $g$ at $x_0$, i.e. $g_{i\bar j}(x_0)=\delta_{i \bar j}$ and $\partial_rg_{i\bar j}(x_0)=0$;

\vspace{0.1cm}
\item $J$ takes its standard form at $x_0$;

\vspace{0.1cm}
\item $g_\phi$ is diagonal at $x_0$;

\vspace{0.1cm}
\item the eigenvalues of $A_\varphi$ regarded as complex matrix are non-increasing. 
\end{itemize} 
Note that in these coordinates since $J$ is parallel with respect to the Levi-Civita connection of $g$, its first derivatives vanish at $x_0$. Moreover, $A_\varphi(x_0)$ takes the following diagonal   expression 
$$
A_\varphi(x_0)=\mathrm{diag}(\lambda_1(x_0),\lambda_1(x_0),\lambda_2(x_0),\lambda_2(x_0),\dots,\lambda_{n}(x_0),\lambda_{n}(x_0))\,.
$$

Let 
$$ 
\tilde A=A_\phi-B\,,
$$ 
where $B$ is a constant matrix of the form 
$$ 
B=\mathrm{diag}(0,0,B_{2\bar 2},B_{2\bar 2},\dots,B_{n\bar n},B_{n\bar n}) 
$$ 
and the components $B_{r\bar r}$ satisfy 
\[
0<B_{2\bar 2}<B_{3\bar 3}<\cdots<B_{n\bar n}<2B_{2\bar 2}\,.
\]
Since the eigenvalues $ \{ \tilde \lambda_1,\dots,\tilde \lambda_n\} $ of $\tilde A$ regarded as  a  quaternionic matrix are distinct at $ x_0 $, they remain distinct in a neighborhood of $x_0$ and so $ \tilde \lambda_1 $ is a smooth function near $x_0$ such that $\tilde \lambda_1(x_0)=\lambda_1(x_0)$. 

\medskip 
We then replace $Q$ with 
$$
\tilde Q=2\sqrt{\tilde \lambda_1}+\alpha(|\nabla \phi |^2_g) + \beta(\phi)
$$
which still achieves a maximum at the point $x_0$.

\medskip 
Let 
$$
L(u):= F^{i\bar j}(A_\varphi)u_{i\bar j} 
$$ 
be the linearized operator of $F$. And note that $ F^{i\bar j}(A_\varphi)$ is diagonal at $ x_0 $. 

From  now on we write $ F^{i\bar j}$ instead of $ F^{i\bar j}(A_\varphi)$ in order to simplify the notation.  We have 
\[
\tilde Q_k(x_0)=0 
\]
and 
$$
L(\tilde Q)=L\left( 2\sqrt{\tilde \lambda_1} \right)+L\left(\alpha(|\nabla \phi |^2_g)\right) + L(\beta(\phi))\leq 0\quad \mbox{ at }x_0\,.
$$
We aim to show that this last inequality implies \eqref{atx0}. Here we handle the three terms of $L(\tilde Q)$ separately.

\medskip 
We have 
\begin{equation*}\label{eq:loglambda}
L\left( 2\sqrt{\tilde \lambda_1} \right)=F^{k\bar k} \left( \frac{\tilde \lambda_{1,k\bar k}}{\sqrt{\lambda_1}} - \frac{|\tilde \lambda_{1,\bar k}|^2}{2\lambda_1\sqrt{\lambda_1}} \right) \quad \mbox{ at }x_0\,,
\end{equation*}
where
\begin{align}\label{eq:C2_der1_lambda}
\tilde \lambda_{1,k}&=\frac{ \partial \tilde\lambda_1}{\partial \tilde A_{i\bar j}} \tilde A_{i\bar j, k}\,, \qquad \tilde \lambda_{1,\bar k}=\frac{ \partial \tilde\lambda_1}{\partial \tilde A_{i\bar j}} \tilde A_{i\bar j,\bar k}\,,\\
\label{eq:C2_der2_lambda}
\tilde \lambda_{1,k\bar k}&=\frac{ \partial^2 \tilde\lambda_1}{\partial \tilde A_{i\bar j}\partial \tilde A_{a\bar b}} \tilde A_{i\bar j, k} \tilde A_{a\bar b,\bar k}+\frac{ \partial \tilde\lambda_1}{\partial \tilde A_{i\bar j}} \tilde A_{i\bar j,k\bar k}\,.
\end{align}
The indexes after the comma denote covariant derivatives with respect to the Levi-Civita connection of $g$. In order to compute the derivatives of $\tilde \lambda_1$ with respect to the entries of the matrix $\tilde A$, since $J$ takes the standard form
\[
J=\begin{pmatrix}
0 & -1 &        &   &\\
1 & 0  &        &   &\\
  &    & \ddots &   &\\
  &    &        & 0 & -1 \\
  &    &        & 1 & 0 \\
\end{pmatrix}
\]
at $ x_0 $, we consider the following basis of hyperhermitian matrices:
$$
\{E_{2r-1\,2s},E_{2r\,2s},E_{2r\,2r}\}\,,\quad  r<s\,,
$$
where 
$$
(E_{2r-1\,2s})_{i\bar j}=\begin{cases}
\,\,\,\, 1 \mbox{ if } (i,j)=(2r-1,2s), \,(2s,2r-1)\\
-1 \mbox{ if } (i,j)=(2r,2s-1),\,(2s-1,2r)\\
\quad\! 0 \mbox{ otherwise}
\end{cases}\quad r<s
$$
and
$$
(E_{2r\,2s})_{i\bar j}=\begin{cases}
\,\,\,\, 1 \mbox{ if } (i,j)=(2r,2s), \,(2s,2r),\,(2r-1,2s-1),\,(2s-1,2r-1)\\
\quad\! 0 \mbox{ otherwise}
\end{cases}\quad r\leq s\,.
$$ 
Since $\tilde A$ is diagonal at $x_0$, for $r<s$ and $ t\in \R $ we have 
\[
\det(\tilde A+tE_{2r-1\,2s}-\lambda \mathrm{Id})=\det(\tilde A+tE_{2r\,2s}-\lambda \mathrm{Id})=\left(\lambda^2-\lambda(\tilde \lambda_r +\tilde \lambda_s)+\tilde \lambda_r \tilde \lambda_s -t^2\right)^2\prod_{\substack{k=1 \\ k\neq r,s}}^{n} (\tilde \lambda_k-\lambda)^2
\]
at $x_0$, 
and
\[
\det(\tilde A+tE_{2r\,2r}-\lambda \mathrm{Id})=\left(\tilde \lambda_r+t-\lambda\right)^2\prod_{\substack{k=1 \\ k\neq r}}^{n} (\tilde \lambda_k-\lambda)^2
\]
at $x_0$. In particular
\[
\lambda_1(\tilde A+tE_{p\,2s})=\begin{cases}
\tilde \lambda_1+t & \text{if } p=2s=2\,,\\
\frac{\tilde \lambda_1+\tilde \lambda_s}{2}+\left(\left(\frac{\tilde \lambda_1-\tilde \lambda_s}{2}\right)^2+t^2\right)^{1/2} & \text{if }p\in \{1,2\}\,,\,s> 1\,,\\
\tilde \lambda_1 & \text{otherwise}
\end{cases}
\]
at $x_0$, where $\lambda_1(\tilde A+tE_{p\,2s})$ denotes the first eigenvalue of $\tilde A+tE_{p\,2s}$.    It follows that
\[
\frac{\partial \tilde \lambda_1}{\partial \tilde A_{i\bar j}}=\delta_{i\bar 1}\delta_{1\bar j}+\delta_{i\bar 2}\delta_{2\bar j}
\]
at $x_0$. 
Similarly, one can compute
\[
\frac{\partial^2 \tilde \lambda_1}{\partial\tilde A_{i\bar j}\partial \tilde A_{a\bar b}}=\begin{cases}
    \frac{2}{\tilde \lambda_1-\tilde \lambda_s} & \text{if } (i,j)=(a,b) \,,\,\, i\in \{1,2\}\,,\,\, j>2\,,\,\,j=2s-1 \text{ or }j=2s\,,\\
    0 & \text{otherwise}
\end{cases}
\]
at $x_0$. Therefore, taking into account that $B_{i\bar j,\bar k}=0$, since $B$ is constant in the chosen neighborhood and the Christoffel symbols of the Levi-Civita connection vanish at $x_0$, equations \eqref{eq:C2_der1_lambda} and \eqref{eq:C2_der2_lambda} reduce at $x_0$ to 
\begin{align}\label{eq:der1_lambda}
\tilde \lambda_{1,k}&
=g^\phi_{1\bar 1, k}+g^\phi_{2\bar 2, k}
\,, \qquad \qquad \tilde \lambda_{1,\bar k}
=g^\phi_{1\bar 1,\bar k}+g^\phi_{2\bar 2,\bar k}\,,\\
\begin{split}\label{eq:der2_lambda}
\tilde \lambda_{1,k\bar k}
&=2\sum_{s>1}\sum_{p=1}^2 \frac{|g^\phi_{2s-1\, \bar p,k}|^2+|g^\phi_{p\, \overline{2s-1}, k}|^2+|g^\phi_{2s\, \bar p,k}|^2+|g^\phi_{p\, \overline{2s}, k}|^2}{\lambda_1-\tilde \lambda_s}
+g^\phi_{1\bar 1,k\bar k}+g^\phi_{2\bar 2,k\bar k}\,.
\end{split}
\end{align}

\medskip 
The next step consists in showing that 
\begin{equation}\label{eq:swap}
\sum_{p=1}^2F^{k\bar k} g^\phi_{p\bar p,k\bar k}\geq  \sum_{p=1}^2F^{k\bar k} g^\phi_{k\bar k,p\bar p} - C\mathcal{F}
\quad \mbox{ at }x_0\,,
\end{equation}
where 
$$
\mathcal{F}=\sum_{k=1}^{2n}F^{k\bar k}(x_0)\,. 
$$
Note that, by \cite[Lemma 9]{Szekelyhidi} we have $ \mathcal{F}>\tau>0 $. Taking into account that $\Gamma_{ab}^c(x_0)=0$, we compute
\[
g^\phi_{p\bar p,k\bar k}=\partial_{\bar p} \partial_p g^\phi_{k\bar k}-\overline{\Gamma_{pk}^r}\partial_pg^\phi_{k\bar r}-\partial_{\bar p}\Gamma^q_{pk}g^\phi_{q\bar k}-\Gamma^q_{pk}\partial_{\bar p}g^\phi_{q\bar k}+\Gamma^q_{pk}\overline{\Gamma^{r}_{pk}}g^\phi_{q\bar r}=\partial_{\bar p} \partial_p g^\phi_{k\bar k}-\partial_{\bar p}\Gamma^q_{pk}g^\phi_{q\bar k}
\]
at $x_0$. Moreover, using that $\Gamma_{pk}^q=-\partial_kJ_p^{\bar s}J_{\bar s}^q$ (cf. \cite[Proof of Lemma 4.3]{BGV2}) and the fact that the first derivatives of $ J $ vanish at $ x_0 $, we deduce
\begin{equation}\label{eq:christ}
\begin{split}
\sum_{p=1}^2\partial_{\bar p}\Gamma^q_{pk}&=-\sum_{p=1}^2 \partial_{\bar p}\partial_pJ^{\bar s}_kJ_{\bar s}^q=-\sum_{p=1}^2 \partial_{\bar p}\partial_kJ^{\bar s}_pJ_{\bar s}^q=-\sum_{p=1}^2 (\partial_{\bar p}\partial_k(J^{\bar s}_pJ_{\bar s}^q)-J^{\bar s}_p\partial_{\bar p}\partial_kJ_{\bar s}^q)\\
&=\sum_{p=1}^2 J^{\bar s}_p\partial_{\bar p}\partial_kJ_{\bar s}^q=\partial_{\bar 1}\partial_kJ_{\bar 2}^q-\partial_{\bar 2}\partial_kJ_{\bar 1}^q=0
\end{split}
\end{equation}
at $x_0$, because $ \partial_{\bar p}J_{\bar q}^a=\partial_{\bar q}J_{\bar p}^a $ in any $ I $-holomorphic coordinate system (see \cite[Remark 2.13]{Dinew-Sroka}). So 
$$
\sum_{p=1}^2g^\phi_{p\bar p,k\bar k}=\sum_{p=1}^2 \partial_{\bar p} \partial_p g^\phi_{k\bar k}
$$
at $x_0$. Since $F^{j\bar k}$ satisfies $ F^{j\bar k}=F^{a \bar b}J_a^{\bar k}J_{\bar b}^j $ we have
\[
\begin{split}
\sum_{p=1}^2F^{k\bar k} \partial_{\bar p} \partial_p g^\phi_{k\bar k} &\geq \frac{1}{2}\sum_{p=1}^2F^{k\bar k}\left(J_k^{\bar b} J_{\bar k}^a  \partial_{\bar p} \partial_p\phi_{a\bar b}+ \partial_{\bar p} \partial_p J_k^{\bar b} J_{\bar k}^a \phi_{a\bar b}+ J_k^{\bar b} \partial_{\bar p} \partial_p J_{\bar k}^a \phi_{a\bar b}\right) -C \mathcal{F}\\
&=\frac{1}{2}\sum_{p=1}^2F^{a\bar a} \partial_{\bar p} \partial_p\phi_{a\bar a}
-\frac{1}{2}\sum_{p=1}^2F^{a\bar a} \partial_{\bar p} \partial_p J_k^{\bar b} J^{k}_{\bar a} \phi_{a\bar b}-\frac{1}{2}\sum_{p=1}^2F^{b\bar b}J^{\bar k}_{b} \partial_{\bar p} \partial_p J_{\bar k}^a \phi_{a\bar b}-C\mathcal{F}
\\
&=
\frac{1}{2}\sum_{p=1}^2F^{a\bar a} \partial_{\bar p} \partial_p\phi_{a\bar a}+\frac{1}{2}\sum_{p=1}^2F^{a\bar a} J_k^{\bar b} \partial_{\bar p} \partial_p J^{k}_{\bar a} \phi_{a\bar b}+\frac{1}{2}\sum_{p=1}^2F^{b\bar b}\partial_{\bar p} \partial_pJ^{\bar k}_{b}  J_{\bar k}^a \phi_{a\bar b}-C\mathcal{F}\\
&=\frac{1}{2}\sum_{p=1}^2F^{a\bar a} \partial_{\bar p} \partial_p\phi_{a\bar a}-C\mathcal{F}
\end{split}
\]
at $ x_0 $, where we used again \eqref{eq:christ}. In a similar way we have 
\[
\sum_{p=1}^2F^{k\bar k} \partial_{\bar k} \partial_k g^\phi_{p\bar p}  \leq \frac{1}{2}\sum_{p=1}^2F^{k\bar k} \partial_{\bar k} \partial_k\phi_{p\bar p} +C\mathcal{F}
\]
at $ x_0 $, and thus we obtain \eqref{eq:swap}.

\medskip
Differentiating the equation $F(A_{\varphi})=h $ we have
\begin{equation}\label{eq:der1_eq}
h_j=F^{i\bar k}g^\phi_{i\bar k,j}=F^{k\bar k}g^\phi_{k\bar k,j}\,,
\end{equation}
and 
\begin{equation}\label{eq:der_eq}
h_{j\bar j}=F^{i\bar k,r\bar s}g^\phi_{i\bar k,j}g^\phi_{r\bar s,\bar j}+F^{k\bar k}g^\phi_{k\bar k,j\bar j} 
\end{equation}
at $x_0$. 
Using \eqref{eq:swap} in \eqref{eq:der2_lambda} and applying \eqref{eq:der_eq} with $j=1,2$ we get
\[
F^{k\bar k} \tilde \lambda_{1,k\bar k} \geq -F^{ik,rs}\sum_{p=1}^2g^\phi_{i\bar k,p}g^\phi_{r\bar s,\bar p}-C\mathcal{F}\geq -C\mathcal{F}
\]
at $x_0$, 
where we also used that $F$ is concave. Hence we have
\begin{equation}\label{prima}
L\left( 2\sqrt{\tilde \lambda_1} \right)\geq -\frac{1}{2\lambda_1\sqrt{\lambda_1}}F^{k\bar k}\Bigl |\sum_{p=1}^2 g^\phi_{p\bar p,\bar k}\Bigr |^2-C\mathcal{F} \quad \mbox{ at } x_0\,.
\end{equation}

\medskip 
Now we handle the second and the third term of $L(\tilde Q)$. 
We have 
\[
\begin{split}
L\left( \alpha(|\nabla\phi|^2_g) \right)&=F^{k\bar k}\left( \alpha''\nabla_{\bar k}|\nabla\phi|^2_g\nabla_k|\nabla\phi|^2_g + \alpha' \nabla_{\bar k}\nabla_k |\nabla\phi|^2_g \right)\\
&=F^{k\bar k}\left( \alpha''\Bigl| \sum_{j=1}^{2n} (\phi_{jk}\phi_{\bar j} + \phi_j\phi_{\bar jk})\Bigr|^2 + \alpha' \sum_{j=1}^{2n} (\phi_{jk\bar k}\phi_{\bar j}+|\phi_{jk}|^2+|\phi_{j\bar k}|^2+\phi_j\phi_{\bar jk\bar k}) \right)\,
\end{split}
\]
at $x_0$. 
Since
\[
\phi_{jk\bar k}=\phi_{k\bar k j}+{R_{k\bar k j}}^q\phi_q\,, \qquad \phi_{\bar j k\bar k}=\phi_{k\bar k \bar j}
\]
and
\[
F^{k\bar k}g^\phi_{k\bar k,j}=F^{k\bar k}\left( \chi_{k\bar k,j}+ \frac{1}{2}\phi_{k\bar k j}+\frac{1}{2}J_k^{\bar b}J_{\bar k}^a\phi_{a\bar b j} \right)\leq F^{k\bar k}\phi_{k\bar k j}+C\mathcal{F}\,,
\]
\[F^{k\bar k}g^\phi_{k\bar k,\bar j}=F^{k\bar k}\left( \chi_{k\bar k,\bar j}+ \frac{1}{2}\phi_{k\bar k \bar j}+\frac{1}{2}J_k^{\bar b}J_{\bar k}^a\phi_{a\bar b \bar j} \right)\leq F^{k\bar k}\phi_{k\bar k \bar j}+C\mathcal{F}
\]
at $ x_0 $, using \eqref{eq:der1_eq} and its conjugate, keeping in mind that $0<\alpha'<(2N)^{-1}$, we obtain 
\[
\alpha' F^{k\bar k}\left( \sum_{j=1}^{2n} (\phi_{jk\bar k}\phi_{\bar j}+\phi_j\phi_{\bar jk\bar k}) \right)\geq -C\mathcal{F}
\]
at $x_0$. Moreover, using the inequality $|a+b|^2\geq \frac{1}{2}|a|^2-|b|^2 $ we have
\[
\begin{split}
\alpha' F^{k\bar k}\sum_{j=1}^{2n} \left(|\phi_{jk}|^2+|\phi_{j\bar k}|^2\right) &\geq \frac{1}{4N} F^{k\bar k} |\phi_{k\bar k}|^2= \frac{1}{8N}F^{k\bar k} (|\phi_{k\bar k}|^2+|J_k^{\bar b}J_{\bar k}^a\phi_{a\bar b}|^2)\\
&=\frac{1}{8N}F^{k\bar k} (|2g^\phi_{k\bar k}-2\chi_{k\bar k}-J_k^{\bar b}J_{\bar k}^a\phi_{a\bar b}|^2+|J_k^{\bar b}J_{\bar k}^a\phi_{a\bar b}|^2)\\
&\geq \frac{1}{4N} F^{k\bar k} |g^\phi_{k\bar k}|^2-C\mathcal{F}
\end{split}
\]
at $x_0$. 
Furthermore, since $\tilde Q_k(x_0)=0$ we infer  
\[
\sum_{p=1}^2\frac{g^\phi_{p\bar p,k}}{\sqrt{\lambda_1}}+\alpha'\sum_{j=1}^{2n}(\phi_{jk}\phi_{\bar j}+\phi_{j}\phi_{\bar jk})+\beta'\phi_k=0
\]
at $x_0$, and using $\alpha''=2(\alpha')^2$, we obtain 
\[
\begin{split}
\alpha'' F^{k\bar k}\Bigl| \sum_{j=1}^{2n} (\phi_{jk}\phi_{\bar j} + \phi_j\phi_{\bar jk})\Bigr|^2&=2F^{k\bar k} \left| \sum_{p=1}^2 \frac{g^\phi_{p\bar p,\bar k}}{\sqrt{\lambda_1}}+\beta'\phi_k \right |^2\\
&\geq \frac{2\epsilon }{\lambda_1}F^{k\bar k}  \Bigl |\sum_{p=1}^2g^\phi_{p\bar p,\bar k}\Bigr |^2-\frac{2\epsilon}{1-\epsilon}(\beta')^2F^{k\bar k}|\phi_k|^2\,,
\end{split}
\]
at $x_0$, where we used the inequality $ |a+b|^2\geq \epsilon |a|^2-\frac{\epsilon}{1-\epsilon}|b|^2 $ for every $\epsilon \in [0,1)$. Hence 
\begin{equation}\label{seconda}
L\left( \alpha(|\nabla\phi|^2_g) \right)\geq \frac{2\epsilon }{\lambda_1}F^{k\bar k}  \Bigl |\sum_{p=1}^2g^\phi_{p\bar p,\bar k}\Bigr |^2-\frac{2\epsilon}{1-\epsilon}(\beta')^2F^{k\bar k}|\phi_k|^2+\frac{1}{4N} F^{k\bar k} |g^\phi_{k\bar k}|^2-C\mathcal{F} \qquad \mbox{ at }x_0\,.
\end{equation}
Moreover, 
\begin{equation}\label{terza}
L\left( \beta(\phi)\right)=\beta' F^{k\bar k}\phi_{k\bar k}+ \beta'' F^{k\bar k}\phi_k \phi_{\bar k}\quad \mbox{ at }x_0\,.
\end{equation}
Then \eqref{prima}, \eqref{seconda}, \eqref{terza} and  $ L(\tilde Q)(x_0)\leq 0$ imply
\begin{equation}\label{eq:Main}
\frac{4\epsilon \sqrt{\lambda_1}-1}{2\lambda_1\sqrt{\lambda_1}}F^{k\bar k}\Bigl |\sum_{p=1}^2 g^\phi_{p\bar p,\bar k}\Bigr |^2+\left(\beta''-\frac{2\epsilon}{1-\epsilon}(\beta')^2\right)F^{k\bar k}|\phi_k|^2+ \frac{1}{4N} F^{k\bar k} |g^\phi_{k\bar k}|^2+\beta' F^{k\bar k}\phi_{k\bar k}-C\mathcal{F}\leq 0 \,.
\end{equation}
at $x_0$. By choosing
\[
\epsilon<\frac{1}{18D^2+1}
\]
and assuming that $\sqrt{\lambda_1}(x_0)\geq \frac{1}{4\epsilon}$ the first two terms on the right hand side of \eqref{eq:Main} are non-negative. Hence
\begin{equation}\label{eq:newMain}
\frac{1}{4N} F^{k\bar k} |g^\phi_{k\bar k}|^2+\beta' F^{k\bar k}\phi_{k\bar k}-C\mathcal{F}\leq 0 	\mbox{ at }x_0\,.  
\end{equation}

Again, we assume $ \underline{\phi}\equiv 0 $, otherwise we could replace $ \chi $ with a suitably chosen form in order to simplify the equation. By definition of $ \mathcal{C} $-subsolution we can find $ \delta,R>0 $ such that
\[
\left( \lambda(g^{-1}\chi)-2\delta {\bf 1}+\Gamma_n \right)\cap \partial \Gamma^{h(x)}\subseteq B_R(0)\,, \qquad \text{at every }x\in M\,.
\]
Suppose $ \lambda_1>R $ then $ |\lambda(A_\phi)|>R $ and we can then apply Lemma \ref{dichotomy} to deduce the existence of a constant $ \kappa>0 $ such that either
\[
-F^{k\bar k}\phi_{k\bar k}>\kappa \mathcal{F}\,,
\]
or
\begin{equation*}
F^{k\bar k}>\kappa \mathcal{F}\,, \qquad \text{for all } k=1,\dots,2n\,.
\end{equation*}
In the first case, choosing $ D $ large enough we can guarantee $ \beta'F^{k\bar k}\phi_{k\bar k}-C\mathcal{F}> 0 $ which, together with \eqref{eq:newMain} yields a contradiction. Therefore, we only need to consider the second case. We have
\[
\beta'F^{k\bar k}\phi_{k\bar k}=\frac{1}{2}\beta'F^{k\bar k}(\phi_{k\bar k}+J_k^ {\bar b}J_{\bar k}^a\phi_{a\bar b})=\beta'F^{k\bar k}(g^\phi_{k\bar k}-\chi_{k\bar k})\geq -3D F^{k \bar k}g^\phi_{k\bar k}-C\mathcal{F}\,,
\]
and so from \eqref{eq:newMain} we obtain
\begin{equation}
\label{eq:newnewMain}
\frac{1}{4N} F^{k\bar k} |g^\phi_{k\bar k}|^2-3DF^{k \bar k}g^\phi_{k\bar k}-C\mathcal{F}\leq 0\,.
\end{equation}
Using $ F^{k \bar k}|g^\phi_{k\bar k}|^2\geq F^{11}|g^\phi_{1\bar 1}|^2=F^{11}\lambda_1^2>\kappa \mathcal{F}\lambda_1^2 $ in \eqref{eq:newnewMain}, we get
\begin{equation}
\label{eq:finale}
\frac{1}{4N}\kappa \mathcal{F}\lambda_1^2-3DF^{k \bar k}g^\phi_{k\bar k}-C\mathcal{F}\leq 0\,.
\end{equation}
Since $ f $ is concave, for $ \lambda=(\lambda_1,\dots,\lambda_n) $ and $ \mathbf{1}=(1,\dots,1) $ we have
\[
f(\lambda)-f(\mathbf{1})\geq \nabla f (\lambda) \cdot (\lambda - \mathbf{1})
\]
which gives, at $ x_0 $
\[
\begin{split}
F^{k \bar k}g^\phi_{k\bar k}&=\sum_{j=1}^n\left(F^{2j-1\, \overline{2j-1}}+F^{2j\, \overline{2j}}\right)\lambda_j=2\sum_{j=1}^n f_j(\lambda) \lambda_j \leq 2f(\lambda)-2f(\mathbf{1})+ 2\sum_{j=1}^n f_j (\lambda)\\
&=2h-2f(\mathbf{1})+\mathcal{F}\leq C\mathcal F
\end{split}
\]
hence, we deduce from \eqref{eq:finale}
\[
\lambda_{1}^2\leq CN\,,
\]
which implies the bound. 
\end{proof}

\section{Gradient estimate}\label{gradient}
In this section we use interpolation inequalities to deduce a bound for the gradient of solutions to \eqref{eq_main}.

\begin{prop}\label{Prop:C1}
Let $ (M^{4n},I,J,K,g) $ be a compact hyperhermitian manifold. Assume that every solution $\phi$ of \eqref{eq_main} satisfies
\[
\Delta_g \phi \leq C(\|\nabla \phi\|_{C^0}+1)\,.
\]
Then there is a bound
\[
\Delta_g \phi \leq C'\,,
\]
depending only on $ (M,I,J,K,g) $, $ C $ and $ \|\phi \|_{C^0} $.
\end{prop}
\begin{proof}
In view of \cite[section 6.8]{GT}, for any $ \epsilon>0 $ and  $ 0<\alpha<1 $ there exists a constant $ C_\epsilon>0 $ such that
\[
\norma{\phi}_{C^1}\leq C_\epsilon \norma{\phi}_{C^0}+\epsilon\norma{\phi}_{C^{1,\alpha}}\,.
\]
By choosing $ p=\frac{4n}{1-\alpha}>4n $ applying Morrey's inequality and elliptic $ L^p $-estimates for the Laplacian, we have 
\[
\norma{\phi}_{C^{1,\alpha}}\leq C_1 \norma{\phi}_{W^{2,p}}\leq C_2\left( \norma{\phi}_{L^p}+\norma{\Delta_g \phi}_{L^p} \right)\leq C_2\left(\norma{\phi}_{C_0}+\norma{\Delta_g \phi}_{C^0}\right)
\]
for some positive constants $ C_1,C_2 $ depending only on $ \alpha $. Therefore, we obtain
\[
\norma{\phi}_{C^1}\leq C_\epsilon \norma{\phi}_{C^0}+\epsilon C_2 C\left( \norma{\phi}_{C^1}+1 \right)\,, 
\]
which gives the gradient bound, and thus the Laplacian bound we were looking for, when we choose $\epsilon<(C_2C)^{-1} $.
\end{proof}

\section{$ C^{2,\alpha} $ estimate}\label{C2alpha}
	
\begin{prop}\label{prop_C2a}
Let $ (M^{4n},I,J,K,g) $ be a compact hyperhermitian manifold. If $ \phi $ is a solution to \eqref{eq_main} such that $ \|\phi\|_{C^0} $ and $ \Delta_g\phi $ are bounded from above by a constant $ C>0 $, then there exists $ \alpha \in (0,1) $ and a constant $ C'>0 $, depending only on the background data and $ C $ such that
\[
\|\phi \|_{C^{2,\alpha}}\leq C'\,.
\]
\end{prop}
\begin{proof}
For any point $ x_0\in M $, take a $I$-holomorphic coordinate chart $(z^1,\dots,z^{2n})$ centered at $ x_0 $ and assume that the domain of the chart contains $ B_1(0) $. Consider also the induced real coordinates $(x^1,\dots,x^{4n})$, where $z^k=x^k+\sqrt{-1}x^{2n+k}$, for $k=1,\dots,2n$. We take into account the real representation of complex matrices $\iota \colon\mathbb \C^{2n,2n}\to \R^{4n,4n} $, defined as
	$$ 
	\iota(H):=
	\begin{pmatrix}
		\Re(H) & \Im(H)\\
		-\Im(H) & \Re(H)
	\end{pmatrix}\,.
 $$
The map $\iota$ sends the space $\mathrm{Herm}(2n)$ of $2n\times 2n$ Hermitian matrices to the space $\mathrm{Sym}(4n)$ of $4n\times 4n$ real symmetric matrices. We will also need the projections $\mathrm{p}\colon \mathrm{Sym}(4n) \to \Im(\iota)$, $T\colon \mathrm{Sym}(4n)\times B_1(0)\to \Im(\iota)$
	$$
	{\rm p}(N):=\frac12(N+{^tI}NI)\,, \qquad T(N,x)=\frac{1}{4}(\mathrm{p}(N)+\iota(J^t(x)) \mathrm{p}(N) \iota(J(x)))\,.
	$$
In the chosen coordinates, for a $ C^2 $-regular function $ u\colon B_1(0)\subseteq \C^{2n} \to \R $ we have 
	\[
	\iota({\rm Hess}_\mathbb{C}u)=\frac{1}{2} {\rm p}(D^2u)\,.
	\]
Observe that, whenever $ N\in \mathrm{Sym}(4n) $, for every $x\in B_1(0)$, the endomorphism
	\[
	\tilde N(x):=g^{-1}(x)(\iota^{-1}(T(N,x)))\,,
	\]
	is hyperhermitian with respect to $ g $. Define the set
	\[
	\mathcal{E}=\left\{ N \in \mathrm{Sym}(4n) \mid \lambda(\tilde N(0)) \in \bar \Gamma^{\sigma}\cap \overline{B_{R}(0)} \right\}\,,
	\]
	where $ \sigma $ and $ R $ are to be chosen. The set $ \mathcal{E} $ is compact and by convexity of $ \Gamma $ it is also convex. Note that by continuity of $ g $, and possibly shrinking $ B_1(0) $ to a smaller radius $ r\in (0,1) $, it follows that $\lambda(\tilde N(x))$ is close to $\lambda(\tilde N(0))$. 
	Possibly shrinking $ B_1(0) $ again we may assume that if $ N $ lies in a sufficiently small neighborhood $ U $ of  $ \mathcal{E} $, then $ \lambda(\tilde N(x))\in \bar \Gamma^{\sigma}\cap \overline{B_{2R}(0)} $ for any $ x\in B_1(0) $.
		
	The bound $ \Delta_g \phi \leq C $ implies that $ \sigma $ and $ R $ can be chosen so that
	\[
	\iota (\chi(x))+T(D^2\phi(x),x)= \iota \left( \chi(x)+ \frac{\mathrm{Hess}_\C \phi(x) +J^t(x)\mathrm{Hess}_\C\phi (x) J(x)}{2} \right)  \in \mathcal{E}\,,
	\]
  for each $ x\in B_1(0) $.
	
Finally, our assumptions on $ f $ ensure that \cite[Theorem 1.2]{TWWY} can be applied with
	\begin{itemize}
		\itemsep0.2em
		\item $ F\colon \mathrm{Sym}(4n)\times B_1(0)\to \R $ defined as $ F(N,x)=f(\lambda(\tilde N(x))) $ for $ N\in U $, and extended smoothly to all of $ \mathrm{Sym}(4n)\times B_1(0) $;
		\item $ S\colon B_1(0)\to \mathrm{Sym}(4n) $ defined as $ S(x)=\iota(\chi(x)) $;
		\item $ T\colon \mathrm{Sym}(4n)\times B_1(0)\to \mathrm{Sym}(4n) $ defined as $ T(N,x)=\frac{1}{4}(\mathrm{p}(N)+\iota(J^t(x)) \mathrm{p}(N) \iota(J(x))) $.
	\end{itemize}
	And since $ \|\phi\|_{C^0}\leq C $ we obtain the desired bound $ \|\phi\|_{C^{2,\alpha}}\leq C $ for some $ \alpha \in (0,1) $.
\end{proof}

\section{Proof of the main result}\label{finale}

\begin{proof}[Proof of Theorem $\ref{main}$]
Let $(M,I,J,K,g)$ be a compact hyperk\"ahler manifold and let $\phi$ be a solution to \eqref{eq_main} such that $\sup_M \phi=0$. By Proposition \ref{Prop:C0} there is a constant $C_1>0$ such that
\[
\|\phi \|_{C^0}\leq C_1\,,
\]
with $C_1$ depending only on background data. We can then apply Proposition \ref{Prop:C2} to get a bound of the form
\[
\Delta_g \phi \leq C_2\left( \|\nabla \phi \|_{C^0} +1 \right)
\]
where $C_2>0$ depends only on the background data. We then deduce a gradient bound from Proposition \ref{Prop:C1} and thus we have
\[
\Delta_g \phi \leq C_3\,.
\]
The $C^{2,\alpha}$ estimate is then implied by Proposition \ref{prop_C2a} and the result follows.
\end{proof}

\end{document}